\documentclass[a4paper,10pt]{article}
\usepackage{amsmath,amsfonts,amssymb,amsthm}
\usepackage{color}
\usepackage[colorlinks=true, linktocpage=true, linkcolor=blue, citecolor=blue]{hyperref}

\theoremstyle{theorem}
\newtheorem{theorem}{Theorem}[section]
\newtheorem{example}[theorem]{Example}
\newtheorem{lemma}[theorem]{Lemma}

\newtheorem{proposition}[theorem]{Proposition}
\newtheorem{conjecture}[theorem]{Conjecture}
\theoremstyle{definition}

\theoremstyle{remark}
\newtheorem{remark}[theorem]{Remark}

\numberwithin{equation}{section}

\begin{document}
\title{\Large Finite-dimensional subalgebras of the Virasoro algebra}
\author{Zhihua Chang\footnote{zhihuachang@gmail.com}}
\maketitle

{\small School of Mathematics, South China University of Technology, Guangzhou, Guangdong, 510640, P. R. China.}

\normalsize

\begin{abstract}
We determine all two-dimensional Lie subalgebras of the centreless Virasoro algebra and complete the characterization of all finite dimensional Lie subalgebras of the complex Virasoro algebra.
\medskip

\medskip

\end{abstract}

\section{Introduction}
Let $\frak{d}$ be the centreless Virasoro algebra over $\Bbb{C}$, which is the Lie algebra of derivations of the Laurent polynomial algebra $\Bbb{C}[t^{\pm1}]$. Obviously, $\frak{d}$ has a basis $\{L_m:=-t^{m}\partial, m\in\Bbb{Z}\}$, where $\partial$ denotes the degree operator $t\frac{d}{dt}$ throughout the paper. They satisfy:
$$[L_m,L_n]=(m-n)L_{m+n}, \text{ for } m,n\in\Bbb{Z}.$$
The one-dimensional non-trivial central extension of $\frak{d}$ is the so-called Virasoro algebra $\hat{\frak{d}}:=\frak{d}\oplus\Bbb{C}K$, on which the bracket is given by
$$[L_m,L_n]=(m-n)L_{m+n}+\frac{1}{12}(m^3-m)\delta_{m,-n}K,$$
for $m,n\in\Bbb{Z}$ and $K$ is a central element.

It has been known for decades that $\frak{d}$ is a simple infinite-dimensional Lie algebra. $\frak{d}$ has no finite dimensional subalgebra of dimension greater than or equal to four (c.f. \cite[Proposition~3.1]{Ng2000witt}). Each three-dimensional subalgebra of $\frak{d}$ is spanned by $\{L_n,L_0,L_{-n}\}$ for some positive integer $n$ (c.f. \cite[Proposition~3.4]{Ng2000witt} or \cite[Lemma~3.1]{suzhao2002virasoro}). However, a complete list of two-dimensional subalgebras of $\frak{d}$ has not been obtained yet. It is easy to observe that $\{L_0,L_n\}$ spans a two-dimensional subalgebra of $\frak{d}$ for each nonzero integer $n$. However, not every two dimensional subalgebra of $\frak{d}$ is of this form. Such examples have been given in \cite[Lemma~3.2]{suzhao2002virasoro}, as well as in \cite{yulu2006virasoro}.

This paper is devoted to determine all two-dimensional subalgebras of $\frak{d}$. Indeed, we have already known that the only commutative subalgebras of $\frak{d}$ are those of one-dimensional. Hence, every two-dimensional subalgebra of $\frak{d}$ is non-commutative, thus has a basis $\{X,Y\}$ such that $[X,Y]=c Y$ for some nonzero $c\in\Bbb{C}$. If we write $X=F(t)\partial$ and $Y=G(t)\partial$ for $F(t), G(t)\in\Bbb{C}[t^{\pm1}]$, then $[X,Y]=c Y$ is equivalent to
\begin{equation}
t(F(t)G'(t)-G(t)F'(t))=c G(t),\label{eq:diffeq}
\end{equation}
where $F'(t)$ and $G'(t)$ are the formal derivatives of $F(t)$ and $G(t)$, respectively. Therefore, our problem that aims to find all two-dimensional subalgebras of $\frak{d}$ is reduced to find all solutions of the differential equation (\ref{eq:diffeq}) in $\Bbb{C}[t^{\pm1}]$. One might use the theory of differential equations to solve (\ref{eq:diffeq}), then to obtain all Laurent polynomial solutions for $F(t)$ and $G(t)$. But we will use algebraic methods to achieve this in this paper.

With this spirit, we will construct a family of two-dimensional subalgebras of $\frak{d}$ in Section~\ref{sec:example}, and discuss properties of the parameters describing this family in Section~\ref{sec:variety}. In Section~\ref{sec:classification}, all two-dimensional subalgebras of $\frak{d}$ will be determined. Finally, we will characterize all finite dimensional subalgebras of the Virasoro algebra $\hat{\frak{d}}$ in Section~\ref{sec:subvir}.

Throughout this paper, we will use $\Bbb{N}, \Bbb{Z}$, and $\Bbb{C}$ to denote the sets of positive integers, integers, and complex numbers, respectively. $\Bbb{Z}^{\times}$ and $\Bbb{C}^{\times}$ will denote the set of nonzero integers and nonzero complex numbers, respectively.

The Lie algebra $\frak{d}$ has a triangular decomposition
$$\frak{d}=\frak{d}_-\oplus\frak{d}_0\oplus\frak{d}_+$$
where $\frak{d}_{\pm}=\mathrm{span}_{\Bbb{C}}\{t^{m}\partial|\pm m\in\Bbb{N}\}$ and $\frak{d}_0=\Bbb{C}\partial$. For an element $X\in\frak{d}$, we write $$X=(\alpha_1t^{r_1}+\cdots+\alpha_st^{r_s})\partial\in\frak{d}$$
such that $r_1<\cdots<r_s$ and $\alpha_1,\ldots,\alpha_s\neq0$. Then we define $\deg_1(X)=r_s$ and $\deg_2(X)=r_1$. A Laurent polynomial $F(t)\in\Bbb{C}[t^{\pm1}]$ is said to be monic if the coefficient of the highest power of $t$ is $1$.

\section{A family of two-dimensional subalgebras of $\frak{d}$}
\label{sec:example}

It is known that $\frak{z}(m)=\mathrm{span}_{\Bbb{C}}\{\partial,t^{m}\partial\}$ is a two-dimensional subalgebra of $\frak{d}$ for each $m\in\Bbb{Z}^{\times}$. The key figure of the subalgebra $\frak{z}(m)$ is that it is contained in either $\frak{d}_{0}\oplus\frak{d}_+$ or $\frak{d}_{0}\oplus\frak{d}_-$. We will create another family of two dimensional subalgebras of $\frak{d}$ such that each two-dimensional subalgebra in the new family is neither contained in $\frak{d}_{0}\oplus\frak{d}_+$, nor contained in $\frak{d}_{0}\oplus\frak{d}_-$.

In order to describe the new family of two-dimensional subalgebras of $\frak{d}$, we first introduce the following notation:
\begin{enumerate}
\item Give $n,k\in\Bbb{N}$ with $n\geqslant k$, we define the set
\begin{equation}
\Gamma(n,k):=\left\{(r_1,\ldots,r_n)\in\Bbb{N}^k\times\{-1\}^{n-k}\middle|
r_1+\cdots+r_n\geqslant k\right\}.\label{eq:conditionr}
\end{equation}
For $\mathbf{r}\in\Gamma(n,k)$, we denote $|\mathbf{r}|:=r_1+\cdots+r_n$.
\item Given $\mathbf{r}=(r_1,\ldots,r_n)\in\Gamma(n,k)$, we define the set
\begin{equation}
V(\mathbf{r}):=\Big\{(a_1,\ldots,a_n)\in\Bbb{C}^n\Big|\textstyle{\sum\limits_{j=1}^n}r_ja_j^i=0,\text{ for } i=1,\ldots,n-1\Big\},\label{eq:defeqvar}
\end{equation}
and denote $V(\mathbf{r})^{\times}:=V(\mathbf{r})\cap(\Bbb{C}^{\times})^n$.
\item $\Sigma:=\{(n,k,\mathbf{r},\mathbf{a})|n,k\in\Bbb{N}\text{ with }n\geqslant k,\mathbf{r}\in\Gamma(n,k),\text{ and }\mathbf{a}\in V(\mathbf{r})^{\times}\}$.
\end{enumerate}

\begin{remark}
For given $n,k\in\Bbb{N}$ with $n\geqslant k$ and $\mathbf{r}\in\Gamma(n,k)$, an element $\mathbf{a}=(a_1,\ldots,a_n)\in V(\mathbf{r})^{\times}$ is an element of $V(\mathbf{r})$ such that all $a_i$ are nonzero. With certain additional restrictions on $\mathbf{r}$, we can prove that a nonzero element\footnote{A nonzero element of $V(\mathbf{r})$ means an element of $V(\mathbf{r})$ with at least one nonzero coordinate.} of $V(\mathbf{r})$ is always an element of $V(\mathbf{r})^{\times}$ (c.f. Proposition~\ref{prop:mul1case}). However, this is not true in general. For example, if $\mathbf{r}=(2,2,-1,-1)\in\Gamma(4,2)$, we have $(1,0,1,1)$ is a nonzero element of $V(\mathbf{r})$, but it is not an element of $V(\mathbf{r})^{\times}$.
\end{remark}
\bigskip

Now, we may proceed to construct two-dimensional subalgebras of $\frak{d}$:

\begin{proposition}
\label{prop:newsubalg}
For $\mu:=(n,k,\mathbf{r},\mathbf{a})\in\Sigma$, let
\begin{align}
P_{\mu}(t)&:=(t-a_1)\cdots(t-a_n)\in\Bbb{C}[t],\label{eq:upoly}\\
Q_{\mu}(t)&:=t^{-|\mathbf{r}|}\cdot(t-a_1)^{r_1+1}\cdots(t-a_k)^{r_k+1}\in\Bbb{C}[t^{\pm1}].\label{eq:vpoly}
\end{align}
Then the two-dimensional subspace
\begin{equation}
\frak{s}(\mu):=\mathrm{span}_{\Bbb{C}}\{P_{\mu}(t)\partial,Q_{\mu}(t)\partial\}\subseteq\frak{d},\label{eq:newsubalg}
\end{equation}
is a Lie subalgebra of $\frak{d}$. Indeed, $P_{\mu}(t)\partial$ and $Q_{\mu}(t)\partial$ satisfy
\begin{equation}
\left[P_{\mu}(t)\partial,Q_{\mu}(t)\partial\right]=c_{\mu} Q_{\mu}(t)\partial,\label{eq:comin2d}
\end{equation}
where $c_{\mu}=(-1)^{n+1}|\mathbf{r}|a_1\cdots a_n$.
\end{proposition}

To prove this proposition, we need the following lemma:
\begin{lemma}
\label{lem:propertya}
Let $n\geqslant k$ be two positive integers and $\mathbf{r}=(r_1,\ldots,r_n)\in\Gamma(n,k)$. Suppose that  $\mathbf{a}=(a_1,\ldots,a_n)\in(\Bbb{C}^{\times})^n$. Then $\mathbf{a}\in V(\mathbf{r})$ if and only if
\begin{equation}
r_i\prod_{j:j\neq i}(a_j-a_i)=|\mathbf{r}|\prod_{j:j\neq i}a_j\label{eq:defeqvar2}
\end{equation}
for $i=1,\ldots,n$. In particular, for $\mathbf{a}\in V(\mathbf{r})^{\times}$, we have $a_i\neq a_j$ for $i\neq j$.
\end{lemma}
\begin{proof}
We first prove that $\mathbf{a}\in V(\mathbf{r})^{\times}$ implies $a_i\neq a_j$ for $i\neq j$.

The set $\{1,\ldots,n\}$ is divided into a disjoint union of subsets $I_1,\ldots, I_s$ according to the equivalence relation: $i\sim j$ if $a_i=a_j$. In order to prove $a_i\neq a_j$ for $i\neq j$, it suffices to show that there are exactly $n$ distinct equivalence classes. For $i=1,\ldots,s$, we use $a_{I_i}$ to denote the common value $a_l$ for $l\in I_i$, and $r_{I_i}=\sum_{l\in I_i}r_l$.

Since $\mathbf{a}\in V(\mathbf{r})$, we have
$$\sum_{j=1}^s r_{I_j}a_{I_j}^i=0$$
for $i=1,\ldots,n-1$.

If $s\leqslant n-1$, then the matrix $(a_{I_j}^i)_{1\leqslant i,j\leqslant s}$ is invertible since $a_{I_j}\neq0$ for $j=1,\ldots,s$ and $a_{I_i}\neq a_{I_j}$ for $i\neq j$. It follows that $r_{I_j}=0$ for all $j=1,\ldots,s$, and hence
$$|\mathbf{r}|=r_1+\cdots+r_n=r_{I_1}+\cdots+r_{I_s}=0,$$
which contradicts the assumption that $|\mathbf{r}|\geqslant k$. Hence, we conclude that $s=n$, i.e., $a_i\neq a_j$ for $i\neq j$.
\medskip

Next, we show that (\ref{eq:defeqvar2}) holds for $i=1,\ldots,n$. Note that
$$\sum_{j=1}^nr_j a_j^i=0$$
holds for $i=1,\ldots,n-1$. For $i=0$, we have $r_1+\cdots+r_n=|\mathbf{r}|$. Hence, we obtain
$$\sum_{j=1}^nr_j a_j^i=\delta_{i,0}|\mathbf{r}|$$
for $i=0,1,\ldots,n-1$. Since $a_j\neq0$ for $j=1,\ldots,n$ and $a_i\neq a_j$ for $i\neq j$, the matrix $(a_j^i)_{0\leqslant i\leqslant n-1\atop 1\leqslant j\leqslant n}$ is invertible. Hence,
\begin{align}
\begin{pmatrix}
r_1\\r_2\\\vdots\\ r_n
\end{pmatrix}
=
\begin{pmatrix}
1&1&\cdots&1\\
a_1&a_2&\cdots&a_n\\
\vdots&\vdots&\vdots&\vdots\\
a_1^{n-1}&a_2^{n-1}&\cdots&a_n^{n-1}
\end{pmatrix}^{-1}
\begin{pmatrix}
|\mathbf{r}|\\0\\\vdots\\0
\end{pmatrix},\label{eq:vondemonde}
\end{align}
which yields that
$$r_i=\frac{\prod_{j:j\neq i}a_j}{\prod_{j:j\neq i}(a_j-a_i)}|\mathbf{r}|,$$
i.e., (\ref{eq:defeqvar2}) holds for $i=1,\ldots,n$.
\bigskip

Conversely, we suppose that $\mathbf{a}\in(\Bbb{C}^{\times})^n$ satisfying (\ref{eq:defeqvar2}). Then it is obvious that $a_i\neq a_j$ for $i\neq j$. It follows that $(a_j^i)_{0\leqslant i\leqslant n-1\atop 1\leqslant j\leqslant n}$ is invertible. Hence, (\ref{eq:vondemonde}) implies that $\mathbf{a}\in V(\mathbf{r})$. This completes the proof.
\end{proof}

Now, we proceed to prove Proposition~\ref{prop:newsubalg}.

\begin{proof}[Proof of the Proposition~\ref{prop:newsubalg}] It suffices to verify the equality (\ref{eq:comin2d}). We first deduce that
$$[P_{\mu}(t)\partial,Q_{\mu}(t)\partial]
=Q_{\mu}(t)
\left(-|\mathbf{r}|\prod_{j=1}^n(t-a_j)+\sum\limits_{l=1}^nr_lt\prod_{j:j\neq l}(t-a_j)\right)\partial.$$
Let
$$F(t):=-|\mathbf{r}|\prod_{j=1}^n(t-a_j)+\sum\limits_{l=1}^nr_lt\prod_{j:j\neq l}(t-a_j).$$
Then
$$F(0)=(-1)^{n+1}|\mathbf{r}|a_1\cdots a_n=:c_{\mu}.$$
On the other hand, by Lemma~\ref{lem:propertya}, we deduce from $\mathbf{a}\in V(\mathbf{r})^{\times}$ that
$$F(a_i)=r_ia_i\prod_{j:j\neq i}(a_i-a_j)=(-1)^{n-1}|\mathbf{r}|a_i\prod_{j:j\neq i}a_j=c_{\mu}.$$
Now, $F(t)$ is a polynomial of degree at most $n$, taking the same value $c_{\mu}$ at $n+1$ distinct points: $0, a_1,\ldots, a_n$. Hence, $F(t)=c_{\mu}$
is a constant number. This completes the proof.
\end{proof}

%
%
%

\begin{proposition}[Uniqueness]
\label{prop:noneq}
\quad\par
\begin{enumerate}
\item For $m,m'\in\Bbb{Z}^{\times}$, $\frak{z}(m)=\frak{z}(m')$ if and only if $m=m'$.
\item For $\mu:=(n,k,\mathbf{r},\mathbf{a})$ and $\mu':=(n',k',\mathbf{r}',\mathbf{a}')\in\Sigma$, the two subalgebras
$$\frak{s}(\mu)=\frak{s}(\mu')$$
if and only if $n=n'$, $k=k'$, and there is a permutation $\sigma$ of $\{1,\ldots,n\}$ such that
\begin{equation}
r_i'=r_{\sigma(i)},\text{ and }a_i'=a_{\sigma(i)},\label{eq:Snact}
\end{equation}
for $i=1,\ldots,n$.
\item For $m\in\Bbb{Z}^{\times}$ and $\mu\in\Sigma$, the two subalgebras $\frak{z}(m)$ and $\frak{s}(\mu)$ are not equal.
\end{enumerate}
\end{proposition}
\begin{proof}
(i) is obvious since $\partial, t^{m}\partial, t^{m'}\partial$ are linear independent if $m\neq m'\in\Bbb{Z}^{\times}$.
\medskip

(ii) Recall that $\frak{s}(\mu)$ (resp. $\frak{s}(\mu')$) has a basis $\{P_{\mu}(t)\partial, Q_{\mu}(t)\partial\}$ (resp. $\{P_{\mu'}(t)\partial$, $Q_{\mu'}(t)\partial\}$). We first claim that
$\frak{s}(\mu)=\frak{s}(\mu')$ if and only if $P_{\mu}(t)=P_{\mu'}(t)$ and $Q_{\mu}(t)=Q_{\mu'}(t)$.

It obvious that $\frak{s}(\mu)=\frak{s}(\mu')$ if $P_{\mu}(t)=P_{\mu'}(t)$ and $Q_{\mu}(t)=Q_{\mu'}(t)$. Conversely, we assume that $\frak{s}(\mu)=\frak{s}(\mu')$. Note that
$Q_{\mu}(t)\partial$ (resp. $Q_{\mu'}(t)\partial$) is a basis of the 1-dimensional derived algebra $[\frak{s}(\mu),\frak{s}(\mu)]$ (resp. $[\frak{s}(\mu'),\frak{s}(\mu')]$) and both $Q_{\mu}(t)$ and $Q_{\mu'}(t)$ are monic. It follows that $Q_{\mu}(t)=Q_{\mu'}(t)$. Since $\frak{s}(\mu)=\frak{s}(\mu')$, there are $\alpha,\beta\in\Bbb{C}$ such that
$$P_{\mu'}(t)\partial=\alpha P_{\mu}(t)\partial+\beta Q_{\mu}(t)\partial.$$
Note that
$$\deg_2(P_{\mu}(t)\partial)=\deg_2(P_{\mu'}(t)\partial)=0,$$
and $\deg_2(Q_{\mu}(t)\partial)=-|\mathbf{r}|\leqslant-k$ since $\mathbf{r}\in\Gamma(n,k)$, we deduce that
$$\deg_2(\alpha P_{\mu}(t)\partial+\beta Q_{\mu}(t)\partial)<0$$
if $\beta\neq0$. This contradicts the fact that $\deg_2(P_{\mu'}(t)\partial)=0$. Hence, $\beta=0$. Now, both $P_{\mu}(t)$ and $P_{\mu'}(t)$ are monic, we obtain that $P_{\mu}(t)=P_{\mu'}(t)$.
\medskip

Next we show that $P_{\mu}(t)=P_{\mu'}(t)$ and $Q_{\mu}(t)=Q_{\mu'}(t)$ if and only if $n=n'$, $k=k'$ and there is a permutation $\sigma$ of $\{1,\ldots,n\}$ such that
$$r_i'=r_{\sigma(i)},\text{ and }a_i'=a_{\sigma(i)},$$
for $i=1,\ldots,n$. This follows from the fact that $n$ (resp. $n'$) is the degree of $P_{\mu}(t)$ (resp. $P_{\mu'}(t)$), $k$ (resp. $k'$) is the number of distinct nonzero roots of $Q_{\mu}(t)$ (resp. $Q_{\mu'}(t)$), $a_1,\ldots,a_n$ (resp. $a'_1,\ldots,a'_{n'}$) are distinct roots of $P_{\mu}(t)$ (resp. $P_{\mu'}(t)$), and $r_i+1$ (resp. $r'_i+1$) is the multiplicity of $a_i$ (resp. $a'_i$) as a root of $Q_{\mu}(t)$ (resp. $Q_{\mu'}(t)$) for $i=1,\ldots,n$.
\medskip

(iii) For $m>0$ (resp. $m<0$), $\frak{z}(m)\subseteq\frak{d}_0\oplus\frak{d}_+$ (resp. $\frak{z}(m)\subseteq\frak{d}_0\oplus\frak{d}_-$). However, for $\mu\in\Sigma$,
$$\deg_2(Q_{\mu}(t)\partial)=-|\mathbf{r}|\leqslant-k,\text{ and }\deg_1(Q_{\mu}(t)\partial)=n>0.$$
Hence, $\frak{s}(\mu)\not\subseteq\frak{d}_0\oplus\frak{d}_{\pm}$, which yields that $\frak{z}(m)$ is not equal to $\frak{s}(\mu)$.
\end{proof}

\section{Classification of two-dimensional subalgebras of $\frak{d}$}
\label{sec:classification}

In this section, we focus on proving that every two-dimensional subalgebra of $\frak{d}$ is exactly equal to one of those given in Section~\ref{sec:example}.

\begin{lemma}[{c.f. Lemma~3.3 of \cite{Ng2000witt}}]
\label{lem:normalsubalg}
Let $\frak{s}$ be a two-dimensional subalgebra of $\frak{d}$. If $\frak{s}\subseteq\frak{d}_0\oplus\frak{d}_+$ (resp. $\frak{s}\subseteq\frak{d}_0\oplus\frak{d}_-$), then
$\frak{s}$ is equal to $\frak{z}(m)$ (resp. $\frak{z}(-m)$) for some positive integer $m$.\qed
\end{lemma}

\begin{lemma}
\label{lem:eigenpoly}
Let $s\in\Bbb{Z}$ and $F(t),G(t)\in\Bbb{C}[t]$ satisfying $F(0)\neq0, G(0)\neq0$. If there is an element $c\in\Bbb{C}^{\times}$ such that
\begin{equation}
[F(t)\partial,t^{s}G(t)\partial]=c\, t^{s}G(t)\partial,\label{eq:eigenpoly1}
\end{equation}
then the following statements hold:
\begin{enumerate}
\item Every root of $G(t)$ is a root of $F(t)$.
\item $F(t)$ has no multiple root.
\item $G(t)$ has no simple root.
\end{enumerate}
\end{lemma}
\begin{proof}
The equation (\ref{eq:eigenpoly1}) is equivalent to
\begin{equation}
sF(t)G(t)+t(F(t)G'(t)-G(t)F'(t))=c G(t).\label{eq:eigenpoly2}
\end{equation}

(i) Suppose $a$ is a root of $G(t)$ of multiplicity $l\geqslant1$. By (\ref{eq:eigenpoly2}), $(t-a)^l|G(t)$ implies that
$$(t-a)^l|tF(t)G'(t).$$
Since the multiplicity of $a\neq0$ in $G(t)$ is $l$, we deduce that $(t-a)^{l-1}|G'(t)$ and $(t-a)^l\not|G'(t)$. Hence, $(t-a)|F(t)$, i.e., $a$ is a root of $F(t)$.
\medskip

(ii) Suppose $a$ is a root of $F(t)$ of multiplicity $l\geqslant2$, and the multiplicity of $a$ in $G(t)$ is $l'\geqslant0$.
Since $(t-a)|F(t)$ and $(t-a)|F'(t)$, the equality (\ref{eq:eigenpoly2}) implies that $(t-a)|G(t)$, i.e., $l'\geqslant1$.

Now, $(t-a)^{l+l'-1}$ divides $F(t)G(t)$, $F'(t)G(t)$ and $F(t)G'(t)$. Applying (\ref{eq:eigenpoly2}) again, we deduce that $(t-a)^{l+l'-1}|G(t)$. Hence, the multiplicity of $a$ in $G(t)$ is at least $l+l'-1>l'$. This is a contradiction.
\medskip

(iii) Suppose $a$ is a simple root of $G(t)$. Then $G(t)=(t-a)G_1(t)$, where $G_1(a)\neq0$. By (ii), $F(t)=(t-a)F_1(t)$, where $F_1(a)\neq0$. We deduce from (\ref{eq:eigenpoly2}) that
$$s(t-a)^2F_1(t)G_1(t)+t(t-a)^2(F_1(t)G_1'(t)-G_1(t)F_1'(t))=c(t-a)G_1(t).$$
It follows that $(t-a)|G_1(t)$, which contradicts that $G_1(a)\neq0$. Hence, $G(t)$ has no simple root.
\end{proof}

\begin{theorem}
\label{thm:2dsubalgofd}
Let $\frak{a}$ be a two-dimensional subalgebra of $\frak{d}$. Then $\frak{a}$ is equal to either $\frak{z}(m)$ for some $m\in\Bbb{Z}^{\times}$, or $\frak{s}(\mu)$ for some $\mu:=(n,k,\mathbf{r},\mathbf{a})\in\Sigma$.
\end{theorem}
\begin{proof}
If $\frak{a}\subseteq\frak{d}_0\oplus\frak{d}_+$ or $\frak{a}\subseteq\frak{d}_0\oplus\frak{d}_-$, then $\frak{a}$ is equal to $\frak{z}(m)$ for some $m\in\Bbb{Z}^{\times}$ (see Lemma~\ref{lem:normalsubalg}). Now, we assume $\frak{a}\not\subseteq\frak{d}_0\oplus\frak{d}_{\pm}$.

Since $\frak{a}$ is a two-dimensional subalgebra of $\frak{d}$, there is a basis $\{X,Y\}$ of $\frak{s}$ such that
$$[X,Y]=c Y,$$
for some non-zero $c\in\Bbb{C}$.

Note that $\{X-\alpha Y,Y\}$ is also a basis of $\frak{a}$ satisfying $[X-\alpha Y,Y]=c Y$. With a suitable choice of $\alpha$, we may assume $\deg_2(X)\neq\deg_2(Y)$. In this situation, $$\deg_2([X,Y])=\deg_2(X)+\deg_2(Y)=\deg_2(Y),$$
which implies that $\deg_2(X)=0$, i.e.,
$$X=F(t)\partial,$$
where $F(t)\in\Bbb{C}[t]$ satisfying $F(0)\neq0$.

We claim that $\deg_1(X)=\deg_1(Y)>0$. We first observe that $\deg_1(X)\geqslant\deg_2(X)=0$. Since $\deg_1(X)=0$ implies that $X=\alpha \partial$, which yields that $\frak{a}=\frak{z}(m)$ for some $m\in\Bbb{Z}^{\times}$, contradicting the assumption that $\frak{a}\not\subseteq\frak{d}_0\oplus\frak{d}_{\pm}$. Hence, $\deg_1(X)>0$. To prove $\deg_1(X)=\deg_1(Y)$, we suppose contrarily that $\deg_1(X)\neq\deg_1(Y)$. Then
$$\deg_1(Y)=\deg_1([X,Y])=\deg_1(X)+\deg_1(Y).$$
Hence, $\deg_1(X)=0$, i.e., $X=\alpha \partial$ for some $\alpha\in\Bbb{C}^{\times}$, which contradicts the assumption that $\frak{a}\not\subseteq\frak{d}_0\oplus\frak{d}_{\pm}$ again. Therefore, the claim follows.

Now, we write
$$Y=t^{s}G(t)\partial$$
such that $G(t)\in\Bbb{C}[t]$ and $G(0)\neq0$. Then
$$[F(t)\partial,t^{s}G(t)\partial]=c\, t^{s}G(t)\partial.$$
By Lemma~\ref{lem:eigenpoly}, we know that $F(t)$ has no multiple root, every root of $G(t)$ is a root of $F(t)$, and $G(t)$ has no simple root. Without losing of generality, we also assume that both $F(t)$ and $G(t)$ are monic. Hence, we write
\begin{align*}
F(t)&=(t-a_1)\cdots(t-a_n),\\
G(t)&=(t-a_1)^{r_1+1}\cdots(t-a_k)^{r_k+1},
\end{align*}
where $n\geqslant1$, $a_1,\ldots,a_n\in\Bbb{C}^{\times}$, and $r_1,\ldots,r_k\in\Bbb{N}$.

Let $\mathbf{r}=(r_1,\ldots,r_k,-1,\ldots,-1)\in\Bbb{N}^k\times\{-1\}^{n-k}$. We deduce from
$$\deg_1(X)=\deg_1(Y)=n$$
that $s=-|\mathbf{r}|$.
\medskip

Next, we will show that $|\mathbf{r}|\geqslant k$. Since $\frak{a}\not\subseteq\frak{d}_0\oplus\frak{d}_{\pm}$ and $X\in\frak{d}_0\oplus\frak{d}_+$, we know that $Y\not\in\frak{d}_0\oplus\frak{d}_+$. Hence, $\deg_2(Y)=-|\mathbf{r}|\leqslant-1$, i.e., $|\mathbf{r}|\geqslant1$. Considering the automorphism of $\frak{d}:$
$$\omega:\frak{d}\rightarrow\frak{d}, \quad t^{l}\partial\mapsto -t^{-l}\partial,$$
we deduce that
\begin{align*}
\omega(X)&=-F(t^{-1})\partial=-t^{-n}(1-a_1t)\cdots(1-a_nt)\partial,\\
\omega(Y)&=-t^{|\mathbf{r}|}G(t^{-1})\partial=-t^{-n}(1-a_1t)^{r_1+1}\cdots(1-a_kt)^{r_k+1}\partial.
\end{align*}
Hence, $\deg_2(\omega(X))=\deg_2(\omega(Y))=-n$, and $\deg_2(\omega(X-Y))>-n$. We further deduce that
$$\deg_2(\omega(Y))=\deg_2([\omega(X-Y),\omega(Y)])=\deg_2(\omega(X-Y))+\deg_2(\omega(Y)).$$
Hence, $\deg_2(\omega(X-Y))=0$. Now,
$$\omega(X-Y)=-t^{-n}(1-a_1t)\cdots(1-a_kt)H(t)\partial,$$
where $H(t)=(1-a_{k+1}t)\cdots(1-a_nt)-(1-a_1t)^{r_1}\cdots(1-a_kt)^{r_k}$. Then $|\mathbf{r}|\geqslant1$ implies that $H(t)$ is a polynomial of degree $r_1+\cdots+r_k$. On the other hand, $\deg_2(\omega(X-Y))=0$ implies that $t^n$ divides $H(t)$, which yields that
$$r_1+\cdots+r_k\geqslant n,$$
i.e., $|\mathbf{r}|=r_1+\cdots+r_k-(n-k)\geqslant k$.
\medskip

Finally, let $\mathbf{a}=(a_1,\ldots,a_n)\in(\Bbb{C}^{\times})^n$. We will show that $\mathbf{a}\in V(\mathbf{r})$. From $[X,Y]=c Y$, we deduce that
\begin{align*}
[X,Y]&=[F(t)\partial,t^{-|\mathbf{r}|}G(t)\partial]\\
&=t^{-|\mathbf{r}|}G(t)\left(-|\mathbf{r}|\prod_{j=1}^n(t-a_j)+\sum_{l=1}^nr_lt\prod_{j:j\neq l}(t-a_j)\right)\partial\\
&=c\, t^{-|\mathbf{r}|}G(t)\partial.
\end{align*}
It follows that
$$C(t):=-|\mathbf{r}|\prod_{j=1}^n(t-a_j)+\sum_{l=1}^nr_lt\prod_{j:j\neq l}(t-a_j)=c$$
is a constant number. Hence,
$$C(a_i)=0+r_ia_i\prod_{j:j\neq i}(a_i-a_j)=c$$
for $i=1,\ldots,n$, and
$$C(0)=(-1)^{n+1}|\mathbf{r}|a_1\cdots a_n=c.$$
It follows that
$$r_i\prod_{j:j\neq i}(a_j-a_i)=|\mathbf{r}|\prod_{j:j\neq i} a_j,$$
for $i=1,\ldots,n$. Since $a_i\neq0$ for $i=1,\ldots, n$, by Lemma~\ref{lem:propertya}, we conclude that $\mathbf{a}\in V(\mathbf{r})^{\times}$. This completes the proof.
\end{proof}

\section{Finite dimensional subalgebras of $\hat{\frak{d}}$}
\label{sec:subvir}

Using the results obtained in the previous sections, we now completely describe all finite dimensional subalgebras of the Virasoro algebra $\hat{\frak{d}}=\frak{d}\oplus\Bbb{C}K$.

\begin{theorem}
Let $\frak{a}$ be a finite dimensional subalgebra of $\hat{\frak{d}}$. Then $\mathrm{dim}(\frak{a})\leqslant4$. Moreover,
\begin{enumerate}
\item If $\dim(\frak{a})=1$, then $\frak{a}=\Bbb{C}X$ for a nonzero $X\in\hat{\frak{d}}$.
\item If $\dim(\frak{a})=2$, then $\frak{a}$ is equal to one of the following subalgebras:
\begin{itemize}
\item $\Bbb{C}X\oplus\Bbb{C}K$ for some nonzero $X\in\frak{d}$, or
\item $\mathrm{span}_{\Bbb{C}}\{L_0+\alpha K, L_m\}$ for some $\alpha\in\Bbb{C}$ and $m\in\Bbb{Z}^{\times}$, or
\item $\mathrm{span}_{\Bbb{C}}\{P_{\mu}\partial+\alpha K, Q_{\mu}\partial+\beta_0 K\}$ for some $\mu\in\Sigma$ and $\alpha\in\Bbb{C}$, where $\beta_0$ is determined by
    \begin{equation}
    [P_{\mu}(t)\partial, Q(\mu)\partial]=\lambda Q_{\mu}(t)\partial+\lambda\beta_0 K\in\hat{\frak{d}}.\label{eq:2dsubcenter}
    \end{equation}
\end{itemize}
\item If $\dim(\frak{a})=3$, then
\begin{itemize}
\item $\frak{a}=\mathrm{span}_{\Bbb{C}}\{L_0+\frac{1}{24}(m^2-1)K, L_{-m}, L_m\}$ for some $m\in\Bbb{Z}^{\times}$, or
\item $\frak{a}=\frak{z}(m)\oplus\Bbb{C}K$ for some $m\in\Bbb{Z}^{\times}$, or
\item $\frak{a}=\frak{s}(\mu)\oplus\Bbb{C}K$ for some $\mu\in\Sigma$.
\end{itemize}
\item If $\dim(\frak{a})=4$, then $\frak{a}=\mathrm{span}_{\Bbb{C}}\{L_0,L_{-m},L_m,K\}$ for some $m\in\Bbb{Z}^{\times}$.
\end{enumerate}
\end{theorem}
\begin{proof}
We consider the canonical homomorphism
$$\pi:\hat{\frak{d}}\rightarrow\frak{d},$$
which maps $X$ to $X$ if $X\in\frak{d}$, and maps $K$ to $0$. Then $\pi(\frak{a})$ is a finite-dimensional subalgebra of $\frak{d}$. Hence, $\dim(\pi(\frak{a}))\leqslant3$. It follows that $\dim(\frak{a})\leqslant4$.

(i) is obvious.
\medskip

(ii) Since $\dim(\frak{a})=2$, $\dim(\pi(\frak{a}))=1$ or $2$. If $\dim(\pi(\frak{a}))=1$, then $\pi(\frak{a})=\Bbb{C}X$ for some nonzero $X\in\frak{d}$. Hence, $\frak{a}=\Bbb{C}X\oplus\Bbb{C}K$. Now we assume $\dim(\pi(\frak{a}))=2$. By Theorem~\ref{thm:2dsubalgofd}, the subalgebra $\pi(\frak{a})=\frak{z}(m)$ for some $m\in\Bbb{Z}^{\times}$ or $\pi(\frak{a})=\frak{s}(\mu)$ for some $\mu\in\Sigma$.

If $\pi(\frak{a})=\frak{z}(m)$, there are $\alpha,\beta\in\Bbb{C}$ such that $\frak{a}=\mathrm{span}_{\Bbb{C}}\{L_0+\alpha K, L_m+\beta K\}$. From
$$[L_0+\alpha K, L_m+\beta K]=-mL_m\in\frak{a},$$
we deduce that $\beta=0$. Hence, $\frak{a}=\mathrm{span}_{\Bbb{C}}\{L_0+\alpha K, L_m\}$ for some $m\in\Bbb{Z}^{\times}$ and $\alpha\in\Bbb{C}$.

If $\pi(\frak{a})=\frak{s}(\mu)$, there are $\alpha,\beta\in\Bbb{C}$ such that
$$\frak{a}=\mathrm{span}_{\Bbb{C}}\{P_{\mu}(t)\partial+\alpha K, Q_{\mu}(t)\partial+\beta K\}.$$
From (\ref{eq:2dsubcenter}), we deduce that $\beta=\beta_0$. Hence,
$$\frak{a}=\mathrm{span}_{\Bbb{C}}\{P_{\mu}(t)\partial+\alpha K,Q_{\mu}(t)\partial+\beta_0 K\}$$
for some $\mu\in\Sigma$ and $\alpha\in\Bbb{C}$.
\medskip

(iii) Since $\dim(\frak{a})=3$, $\dim(\pi(\frak{a}))=2$ or $3$. If $\dim(\pi(\frak{a}))=2$, by Theorem~\ref{thm:2dsubalgofd}, $\pi(\frak{a})=\frak{z}(m)$ for some $m\in\Bbb{Z}^{\times}$ or $\pi(\frak{a})=\frak{s}(\mu)$ for some $\mu\in\Sigma$. Hence, $\frak{a}$ is $\frak{z}(m)\oplus\Bbb{C}K$ or $\frak{s}(\mu)\oplus\Bbb{C}K$. Now, we assume $\dim(\pi(\frak{a}))=3$. Then $\pi(\frak{a})=\mathrm{span}_{\Bbb{C}}\{L_{-m},L_0,L_m\}$ for some $m\in\Bbb{Z}^{\times}$. It follows that
$$\frak{a}=\mathrm{span}_{\Bbb{C}}\{L_{-m}+\alpha K,L_0+\beta K,L_m+\gamma K\}$$
for some $\alpha,\beta,\gamma\in\Bbb{C}$. Note that $\frak{a}$ is a three dimensional subalgebra of $\frak{d}$, we further deduce that $\alpha=\gamma=0$ and $\beta=\frac{1}{24}(m^2-1)$. Hence,
$$\frak{a}=\mathrm{span}_{\Bbb{C}}\{L_{-m},L_0+\textstyle{\frac{1}{24}}(m^2-1) K,L_m\}.$$
\medskip

(iv) has been proved in \cite[Corollary~3.5]{Ng2000witt}.
\end{proof}

\section{Further discussion on the algebraic set $V(\mathbf{r})$}
\label{sec:variety}

To create a two-dimensional subalgebra $\frak{s}(\mu)$, it suffices to give a quadruple $(n,k,\mathbf{r},\mathbf{a})$, where $n,k\in\Bbb{N}$ with $n\geqslant k$, $\mathbf{r}\in\Gamma(n,k)$, and $\mathbf{a}\in V(\mathbf{r})^{\times}$. It is easy to observe that $V(\mathbf{r})$ is an algebraic set solely depending on $\mathbf{r}$. However, for an arbitrary $\mathbf{r}\in\Gamma(n,k)$, a concrete parametrization for all points of $V(\mathbf{r})$ is not known. Nonetheless, we can describe all points of $V(\mathbf{r})^{\times}$ in a few special cases and estimate the cardinality of the set $V(\mathbf{r})^{\times}$ for an arbitrary $\mathbf{r}\in\Gamma(n,k)$. These will be the main issues discussed in this section.

We first create a few concrete examples.

\begin{example}Let $n=k=1$. Then $\mathbf{r}=r$ could be an arbitrary positive integer, and
$$V(\mathbf{r})=\Bbb{C}^{\times}.$$
We obtain two-dimensional subalgebras of $\frak{d}$:
$$\mathrm{span}_{\Bbb{C}}\{(t-a)\partial, t^{-r}(t-a)^{r+1}\partial\}$$
for $r\in\Bbb{N}$ and $a\in\Bbb{C}^{\times}$.
\end{example}

\begin{example}
Let $n=2$ and $k=1$ or $2$. Then $\mathbf{r}=(r_1,r_2)$ with $r_1,r_2\in\{-1\}\cup\Bbb{N}$ such that $r_1+r_2\geqslant k$. In this situation,
$$V(\mathbf{r})^{\times}=\{(ar_2,-ar_1)|a\in\Bbb{C}^{\times}\}.$$
It yields two-dimensional subalgebras of $\frak{d}$:
$$\mathrm{span}_{\Bbb{C}}\{(t-ar_2)(t+ar_1)\partial,t^{-r_1-r_2}(t-ar_2)^{r_1+1}(t-ar_1)^{r_2+1}\partial\}$$
for $r_1,r_2\in\{-1\}\cup\Bbb{N}$ satisfying $r_1+r_2\geqslant k$, and $a\in\Bbb{C}^{\times}$.
\end{example}

\begin{example}
Let $n=3$, $1\leqslant k\leqslant3$ and $\mathbf{r}=(r_1,r_2,r_3)\in\Gamma(n,k)$ such that $r_1\geqslant r_2\geqslant r_3$.
\begin{itemize}
\item If $(r_2,r_3)\neq(1,-1)$, then
$$V(\mathbf{r})^{\times}=\{a(-r_3\pm\textstyle{\frac{1}{r_1}}\sqrt{-r_1r_2r_3|\mathbf{r}|},-r_3\mp\textstyle{\frac{1}{r_2}}\sqrt{-r_1r_2r_3|\mathbf{r}|},r_1+r_2)|a\in\Bbb{C}^{\times}\},$$
where $|\mathbf{r}|=r_1+r_2+r_3$.
\item If $(r_2,r_3)=(1,-1)$, then
$$V(\mathbf{r})^{\times}=\{a(2,1-r_1,r_1+1)|a\in\Bbb{C}^{\times}\}.$$
\end{itemize}
\end{example}

For $n\geqslant4$, we have the following examples:

\begin{example}
Let $n=k$ be an arbitrary positive integer and $\mathbf{r}=(r,\ldots,r)$ for $r\in\Bbb{N}$. Then $\mathbf{a}=(\zeta_n,\zeta_n^2,\ldots,\zeta_n^n)\in V(\mathbf{r})^{\times}$, where $\zeta_n$ is a primitive $n$-th root of unity. The two-dimensional Lie algebra $\frak{s}(\mu)$ for $\mu=(n,n,\mathbf{r},\mathbf{a})$ is
$$\mathrm{span}_{\Bbb{C}}\{(t^n-1)\partial, t^{-rn}(t^n-1)^{r+1}\partial\}.$$
\end{example}

\begin{example}
Let $\mu:=(n,k,\mathbf{r},\mathbf{a})\in\Sigma$. Then
\begin{enumerate}
\item $(n,k,\mathbf{r},c\mathbf{a})\in\Sigma$ for all $c\in\Bbb{C}^{\times}$.
\item If $n=k$, then $\mathbf{r}=(r_1,\ldots,r_n)$ with $r_i\in\Bbb{N}$ for $i=1,\ldots,n$. In this situation, $\mathbf{a}\in V(s\mathbf{r})^{\times}$ for all $s\in\Bbb{N}$. Hence, $(n,n,s\mathbf{r},\mathbf{a})\in\Sigma$.
\end{enumerate}
\end{example}

\begin{example}
For each $s\in\Bbb{N}$,
$$\tau_s:\frak{d}\rightarrow\frak{d},\quad t^{l}\partial\mapsto st^{sl}\partial$$
is an injective homomorphism. Hence, if $$\frak{a}=\mathrm{span}_{\Bbb{C}}\{F(t)\partial,G(t)\partial\}$$ is a two-dimensional subalgebra of $\frak{d}$. Then
$$\mathrm{span}_{\Bbb{C}}\{F(t^s)\partial, G(t^s)\partial\}$$
is also a two-dimensional subalgebra of $\frak{d}$.

In particular, if $\frak{a}=\frak{s}(\mu)$ with $\mu=(n,k,\mathbf{r},\mathbf{a})$, then we obtain the two-dimensional subalgebra $\frak{s}(\mu')$ with $\mu'=(sn,sk,\mathbf{r}',\mathbf{a}')$, where
$$\mathbf{r}'=(\underbrace{r_1,\ldots,r_1}_{s\text{ copies}},\ldots,\underbrace{r_n,\ldots,r_n}_{s\text{ copies}}),\text{ and }\mathbf{a}'=(a_{1,1},\ldots,a_{1,s},\ldots,a_{n,1},\ldots,a_{n,s}),$$
in which $a_{i,1},\ldots,a_{i,s}$ are $s$ distinct roots of $t^s-a_i$ for each $i=1,\ldots,n$.
\end{example}

In general, we observe that the definition equations of $V(\mathbf{r})$ in (\ref{eq:defeqvar}) are homogeneous, and hence define a projective variety $\overline{V}(\mathbf{r})\subseteq\Bbb{P}^{n-1}(\Bbb{C})$. We view $\Bbb{P}^{n-1}(\Bbb{C})$ as $\Bbb{C}^n/\Bbb{C}^{\times}$ and denotes the image of $V(\mathbf{r})^{\times}$ in $\Bbb{P}^{n-1}(\Bbb{C})$ by $\overline{V}(\mathbf{r})^{\times}$.

\begin{lemma}
\label{lem:pvar2}
Let $\mathbf{a}\in V(\mathbf{r})^{\times}$ and $\bar{\mathbf{a}}$ the canonical image of $\mathbf{a}$ in $\overline{V}(\mathbf{r})$. Then $\bar{\mathbf{a}}$ has multiplicity $1$ in $\overline{V}(\mathbf{r})$.
\end{lemma}
\begin{proof}
Note that the Jacobian matrix of the defining equations in (\ref{eq:defeqvar}) is
$$J(x)=(ir_jx_j^{i-1})_{1\leqslant i\leqslant n-1\atop 1\leqslant j\leqslant n}$$
Evaluating at $\mathbf{a}\in V(\mathbf{r})^{\times}$, its sub-matrix consisting of the first $n-1$ columns has the determinant
$$(n-1)!\left(\prod_{j=1}^{n-1}r_j\right)\cdot\left(\prod_{1\leqslant i\neq j\leqslant n-1}(a_i-a_j)\right)$$
which is nonzero by Lemma~\ref{lem:propertya}, i.e., the Jacobian matrix has rank $n-1$ at $\mathbf{a}$. Hence, $\bar{\mathbf{a}}$ has multiplicity one.
\end{proof}

\begin{proposition}
\label{prop:finite}
The set $\overline{V}(\mathbf{r})^{\times}$ contains at most $(n-1)!$ elements.
\end{proposition}
\begin{proof}
It follows from Lemma~\ref{lem:pvar2} that every element $\bar{\mathbf{a}}\in\overline{V}(\mathbf{r})^{\times}$ has multiplicity one in the projective algebraic set $\overline{V}(\mathbf{r})$. It follows that every element $\bar{\mathbf{a}}$ forms an irreducible component of $\overline{V}(\mathbf{r})$. Hence, the number of elements in $\overline{V}(\mathbf{r})^{\times}$ does not exceed the number of irreducible components of $\overline{V}(\mathbf{r})$. By the refined B\'{e}zout's theorem, this number is less than or equal to $(n-1)!$ (c.f. \cite[12.3]{fulton1998}).
\end{proof}

\begin{remark}
Given $n,k\in\Bbb{N}$ with $n\geqslant k$ and $\mathbf{r}\in\Gamma(n,k)$, we know from Proposition~\ref{prop:finite} that $\overline{V}(\mathbf{r})^{\times}$ is finite, but $\overline{V}(\mathbf{r})$ is not necessarily finite. For example, if  $\mathbf{r}=(4,1,1,-1,-1)\in\Gamma(5,3)$, then $(0,1,a,1,a)$ with $a\in\Bbb{C}$ represent infinitely many elements in $\overline{V}(\mathbf{r})$.
\end{remark}

\begin{proposition}
\label{prop:mul1case}
Let $n,k\in\Bbb{N}$ with $n\geqslant k$ and $\mathbf{r}=(r_1,\ldots,r_k,-1,\ldots,-1)\in\Gamma(n,k)$. If
\begin{equation}
r_i\geqslant n-k+1\label{eq:strongconditionr}
\end{equation}
for all $i=1,\ldots,k$. Then $\overline{V}(\mathbf{r})^{\times}$ has exactly $(n-1)!$ elements.
\end{proposition}

In order to prove this proposition, we need the following lemma.

\begin{lemma}
\label{lem:mul1case}
Let $n,k\in\Bbb{N}$ with $n\geqslant k$ and $\mathbf{r}\in\Gamma(n,k)$ satisfying (\ref{eq:strongconditionr}). Then every element of $\overline{V}(\mathbf{r})$ is an element of $\overline{V}(\mathbf{r})^{\times}$.
\end{lemma}
\begin{proof}
It suffices to show that every nonzero element $\mathbf{a}$ of $V(\mathbf{r})$ is contained in $V(\mathbf{r})^{\times}$.

Let $(0,\ldots,0)\neq\mathbf{a}=(a_1,\ldots,a_n)\in V(\mathbf{r})$. We shall show that $a_i\neq0$ for all $i=1,\ldots,n$. As we did in Lemma~\ref{lem:propertya}, we divide $\{1,\ldots,n\}$ into the disjoint union of the equivalence classes $I_1,\ldots,I_s$ according to the equivalence relation: $i\sim j$ if $a_i=a_j$. We denote $a_{I_j}$ the common value $a_l$ for $l\in I_j$ and $r_{I_j}=\sum_{l\in I_j}r_l$. Then $\mathbf{a}\in V(\mathbf{r})$ implies
$$\sum_{j=1}r_{I_j}a_{I_j}^i=0$$
for $i=1,\ldots,n$. Suppose contrarily that $a_j=0$ for some $j=1,\ldots,n$. Without losing of generality, we may assume $a_{I_s}=0$. Then $a_{I_1},\ldots,a_{I_{s-1}}$ are distinct nonzero numbers, which implies that the matrix $(a_{I_j}^i)^{1\leqslant i,j\leqslant s-1}$ is invertible. It follows that $r_{I_1}=\cdots=r_{I_{s-1}}=0$. However, since $\mathbf{r}$ satisfies (\ref{eq:strongconditionr}), there is no subset of $\{r_1,\ldots,r_n\}$ with summation zero. Hence, $I_1,\ldots,I_{s-1}$ are all empty sets, i.e., $\mathbf{a}=(0,\ldots,0)$ which contradicts the assumption. Therefore, $a_i\neq0$ for all $i=1,\ldots,n$.
\end{proof}

\begin{proof}[Proof of Proposition~\ref{prop:mul1case}]
By Lemma~\ref{lem:mul1case}, every element $\mathbf{a}$ of $\overline{V}(\mathbf{r})$ is an element of $\overline{V}(\mathbf{r})^{\times}$. Hence, it follows from Lemma~\ref{lem:pvar2} that $\mathbf{a}$ has multiplicity one in $\overline{V}(\mathbf{r})$. Therefore, $\overline{V}(\mathbf{r})$ contains only finitely many points since every point is an isolated point and $\overline{V}(\mathbf{r})$ has only finitely many irreducible components. Using B\'{e}zout's theorem (see Proposition~8.4 of \cite{fulton1998}), we deduce that $\overline{V}(\mathbf{r})$ has $(n-1)!$ points counting multiplicity. Now, every point is of multiplicity one. Hence, $\overline{V}(\mathbf{r})$ has exactly $(n-1)!$ points, i.e., $\overline{V}(\mathbf{r})^{\times}$ has exactly $(n-1)!$ points.
\end{proof}

In general, if the condition (\ref{eq:strongconditionr}) is not satisfied, we do not have a formula for describing the number of elements in $\overline{V}(\mathbf{r})^{\times}$. Based on computational results using Maple (a computer algebra system), we conjecture that

\begin{conjecture}
Let $n,k\in\Bbb{N}$ with $n\geqslant k$ and $\mathbf{r}\in\Gamma(n,k)$. Then $V(\mathbf{r})^{\times}$ is non-empty.
\end{conjecture}

Maple shows this is true for all $(n,k,\mathbf{r})$ such that $n=4,\ldots,9$, $1\leqslant k\leqslant n$ and $1\leqslant r_i\leqslant n-k$ for $i=1,\ldots,k$.

\section*{Acknowledgement}
The author is grateful to the Azrieli foundation for the award of an Azrieli fellowship. He acknowledges helpful comments by the referees. He also appreciates valuable advices by Boris Kunyavskii, and useful suggestions by Jason Starr and Francesco Polizzi through the MathOverflow.

\end{document}